\documentclass{opt2026} %

\usepackage{bm}
\usepackage{mathtools}
\usepackage{booktabs}
\usepackage{enumitem}
\usepackage{xcolor}
\usepackage[normalem]{ulem}
\usepackage{graphicx}
\usepackage{subcaption}

\usepackage{float}

\newcommand{\R}{\mathbb{R}}
\newcommand{\ip}[2]{\langle #1,#2\rangle}
\newcommand{\norm}[1]{\lVert #1\rVert}
\newcommand{\dist}{\operatorname{dist}}
\newcommand{\argmax}{\operatorname*{arg\,max}}
\newcommand{\proj}{\operatorname{proj}}

\newcommand{\tighteq}{\mkern-2mu=\mkern-2mu}
\newcommand{\tightminus}{\mkern-2mu-\mkern-2mu}

\title[Primal-Dual Adaptive Subgradient Methods]{Primal-Dual Adaptive Subgradient Methods}

\optauthor{%
\Name{Karthik Prakhya} \Email{karthik.prakhya@umu.se}\\
\Name{Zimeng Wang} \Email{zimeng.wang@umu.se}\\
\Name{Alp Yurtsever} \Email{alp.yurtsever@umu.se}\\
\addr Ume{\aa} University, Sweden}

\begin{document}

\maketitle

\begin{abstract}
    We propose adaptive primal-dual subgradient methods for constrained convex optimization.
    Our first method applies an AdaGrad-style step-sizes to the dual problem and recovers primal solutions by averaging the associated Fenchel-type oracle outputs.
    It achieves $\widetilde{\mathcal{O}}(T^{-1/2})$ convergence for the dual objective residual, primal objective residual, and feasibility gap in the nonsmooth regime, while automatically improving to $\mathcal{O}(T^{-1})$ when the dual objective is smooth.
    We further develop an accelerated variant, which retains the $\widetilde{\mathcal{O}}(T^{-1/2})$ nonsmooth rate while achieving $\mathcal{O}(T^{-2})$ convergence in the smooth regime under mild assumptions.
    These guarantees require no prior knowledge of the underlying regularity.
    We demonstrate the empirical performance of the proposed methods on optimal transport and semidefinite relaxation of the Max-Cut problem.
\end{abstract}

\section{Introduction}
\label{sec:introduction}

Consider convex constrained optimization problems of the form
\begin{equation}
\label{eqn:problem-setup}
\tag{P}
    \min_{\bm{x}\in\mathcal{X}} f(\bm{x}) \quad \mathrm{s.t.} \quad \bm{A}\bm{x}-\bm{b}\in\mathcal{K},
\end{equation}
where $f:\R^p\to\R\cup\{+\infty\}$ is proper, closed, and convex;  $\mathcal{X}\subseteq\R^p$ and $\mathcal{K}\subseteq\R^n$ are convex and compact sets; $\bm{A}\in\R^{n\times p}$ and $\bm{b}\in\R^n$. 
Primal-dual subgradient methods provide a simple and scalable approach to this class of problems, as their iterations typically rely only on inexpensive first-order and Fenchel-type oracle evaluations. 
Beyond their low per-iteration arithmetic cost, they are also amenable to memory-efficient implementations based on sketching and dimensionality reduction when primal recovery is performed through averaging \citep{yurtsever2021scalable}. 

A central challenge in practice for subgradient methods is the choice of step-size. 
Classical convergence guarantees typically rely on a decreasing step-size schedules designed for the worst-case nonsmooth regime. 
Nonsmooth convex optimization is a broad class however, and individual instances can exhibit substantially more favorable regularity than the worst-case model.
Adaptive methods such as AdaGrad \cite{duchi2011adaptive,levy2017online}, instead determine their step-sizes from the observed gradients and can therefore exploit additional smoothness when present, achieving faster convergence rates without prior knowledge of the underlying smoothness degree.

This motivates the central question of this paper:
\begin{center}
\textit{Can AdaGrad-style step-sizes be used in the dual space while still providing \\ adaptive guarantees for the recovered primal objective residual and feasibility gap?}
\end{center}
Our answer is affirmative, for both basic and accelerated primal-dual constructions.
We apply AdaGrad-type step-sizes to the unconstrained dual formulation of \eqref{eqn:problem-setup}, while recovering primal iterates from the associated Fenchel-type oracles via averaging. 
Our first method, \emph{Primal–Dual AdaGrad}, simultaneously achieves $\widetilde{\mathcal{O}}(T^{-1/2})$ convergence of the dual objective residual, primal objective residual, and feasibility gap for nonsmooth problems, and automatically improves to $\mathcal O(T^{-1})$ when the dual objective is smooth. 
We further develop an accelerated variant, \emph{Primal-Dual AcceleGrad}, which retains the $\widetilde{\mathcal O}(T^{-1/2})$ nonsmooth guarantee while achieving $\mathcal O(T^{-2})$ rates in the smooth regime. 

\paragraph{Related works.}

Duchi et al.~\citep{duchi2011adaptive} introduced AdaGrad, pioneering adaptive step sizes based on accumulated gradient norms. 
While originally developed for online and stochastic optimization, this mechanism has subsequently been used to obtain universal methods that adapt to different smoothness and noise regimes~\citep{levy2017online}. 
In particular, AcceleGrad~\citep{levy2018online} combines AdaGrad-style adaptation with acceleration, attaining guarantees for both smooth and nonsmooth convex optimization without prior knowledge of the smoothness regime or its constants.
AdaGrad-type adaptation has also been extended to constrained optimization using projection and proximal operations. 
Kavis et al.~\citep{kavis2019unixgrad} proposed UniXGrad for constrained convex optimization, while Bach and Levy~\citep{bach2019universal} developed a universal Mirror-Prox method for variational inequalities. 
Both methods use extra-gradient-type updates and adapt to smooth and nonsmooth regimes as well as deterministic and stochastic oracles. 
Antonakopoulos et al.~\citep{antonakopoulos2021adaptive} further developed adaptive extra-gradient methods for min-max optimization and games. 
Ene et al.~\citep{ene2021adaptive} developed universal adaptive methods for constrained convex optimization based on adaptive projected updates, including accelerated variants, while Wang and Yurtsever~\citep{wang2026universal} developed universal adaptive methods for composite objectives. 
In contrast, our methods address constrained problems through their unconstrained dual and require only Fenchel-type oracles, rather than projections onto or proximal updates over the constraint set.

A complementary approach to universality originates with Nesterov~\citep{nesterov2015universal}. 
His universal gradient methods adapt to unknown H\"older smoothness regime and constants through an inexact line-search strategy. 
Yurtsever et al.~\citep{yurtsever2015universal} extended this approach to constrained convex optimization using a primal-dual framework with Fenchel-type oracles, providing guarantees for both the objective residual and feasibility gap. 
Their methods, however, rely on an inexact line-search strategy that requires objective evaluations and can be costly or impractical when such evaluations are expensive. 
Moreover, they require a target accuracy to be specified in advance and converge only to a neighborhood determined by this accuracy, making it difficult to run them to arbitrary accuracy without prior knowledge of the convergence regime. 
Our methods retain the same primal-dual oracle framework while avoiding line searches, objective evaluations, and a prescribed target accuracy, and adapt automatically to the underlying smoothness regime. 

\section{Algorithms and Main Results}
\label{sec:algorithms-results}

In this section, we present our adaptive primal-dual methods and their convergence guarantees. 
Our goal is to approximately solve \eqref{eqn:problem-setup} to obtain $\bm{x}_\epsilon$ in the following sense:
\begin{definition}
Let $f_\star$ denote the optimal value of \eqref{eqn:problem-setup}. 
Given an accuracy level $\epsilon > 0$, a point $\bm{x}_\epsilon \in \mathcal{X}$ is said to be an $\epsilon$-solution of \eqref{eqn:problem-setup} if
\begin{equation*}
f(\bm{x}_\epsilon) - f_\star \leq \epsilon, \quad \text{and} \quad \mathrm{dist}(\bm{A}\bm{x}_\epsilon - \bm{b}, \mathcal{K}) \leq \epsilon.
\end{equation*}
We refer to $|f(\bm{x}) \tightminus f_\star|$ as the primal objective residual and to $\mathrm{dist}(\bm{A}\bm{x}\tightminus\bm{b},\mathcal{K})$ as feasibility gap.
\end{definition}

Our methods apply adaptive subgradient schemes \cite{levy2017online,levy2018online} to an unconstrained dual formulation of \eqref{eqn:problem-setup}. 
To derive this formulation, we introduce a slack variable $\bm{r}\in\mathcal{K}$ and equivalently write \eqref{eqn:problem-setup} as
\begin{equation}
    \label{eq:problem-slack}
    \min_{\bm{x}\in\mathcal{X},\,\bm{r}\in\mathcal{K}} f(\bm{x}) \quad \mathrm{s.t.} \quad \bm{A}\bm{x}-\bm{r}=\bm{b}.
\end{equation}
The resulting dual problem (written in minimization form for convenience) is
\begin{equation}\label{dual-problem-setup} \tag{D}
    \min_{\bm{\lambda}\in\R^n} g(\bm{\lambda}) := 
    \underbrace{\max_{\bm{x}\in\mathcal{X}}\{\ip{\bm{\lambda}}{-\bm{A}\bm{x}}-f(\bm{x})\}}_{g_x(\bm{\lambda})}
    + \ip{\bm{\lambda}}{\bm{b}} 
    + \underbrace{\max_{\bm{r}\in\mathcal{K}}\ip{\bm{\lambda}}{\bm{r}}}_{g_r(\bm{\lambda})}.
\end{equation}
Throughout the paper, we assume that the primal problem is feasible, strong duality holds (i.e., $f_\star = -g_\star$), and the dual optimum is attained. We denote a dual optimal solution by $\bm{\lambda}_\star$.

Our methods construct dual subgradients through Fenchel-type oracle evaluations. 
For a given dual point $\bm{\lambda}$, let 
\begin{equation*}
    \bm{x}^*(\bm{\lambda})\in\argmax_{\bm{x}\in\mathcal{X}}\{\ip{\bm{\lambda}}{-\bm{A}\bm{x}}-f(\bm{x})\}, \qquad 
    \bm{r}^*(\bm{\lambda})\in\argmax_{\bm{r}\in\mathcal{K}}\ip{\bm{\lambda}}{\bm{r}}
\end{equation*}
denote arbitrary maximizers. 
Then, by Danskin's theorem, $-\bm{A}\bm{x}^*(\bm{\lambda})+\bm{b}+\bm{r}^*(\bm{\lambda}) \in \partial{g}(\bm{\lambda})$. 

\subsection{Primal-Dual AdaGrad}
\label{subsec:pd-adagrad}

Our first algorithm applies a scalar AdaGrad step-size to the unconstrained dual objective and recovers primal solutions by averaging the Fenchel-type oracle outputs, as shown in Algorithm~\ref{alg:pd-adagrad}.

\begin{algorithm2e}[t]
\caption{Primal-Dual AdaGrad}
\label{alg:pd-adagrad}
\KwIn{Initial dual point $\bm{\lambda}_0$, scale parameter $\eta>0$, accumulator offset $\delta>0$.}
\For{$t=0,1,\ldots,T-1$}{
Compute $\bm{x}^*_t=\bm{x}^*(\bm{\lambda}_t)$ and $\bm{r}^*_t=\bm{r}^*(\bm{\lambda}_t)$\;
Set $\bm{g}_t=-\bm{A}\bm{x}^*_t+\bm{b}+\bm{r}^*_t$\;
Set $S_t=\delta+\sum_{\tau=0}^{t}\norm{\bm{g}_\tau}^2$ and $\eta_t=\eta/\sqrt{S_t}$\;
Update $\bm{\lambda}_{t+1}=\bm{\lambda}_t-\eta_t\bm{g}_t$\;
}
\KwOut{$\bar{\bm{\lambda}}_T=A_T^{-1}\sum_{t=0}^{T-1}\eta_t\bm{\lambda}_t$ and $\bar{\bm{x}}_T=A_T^{-1}\sum_{t=0}^{T-1}\eta_t\bm{x}^*_t$, where $A_T=\sum_{t=0}^{T-1}\eta_t$.}
\end{algorithm2e}

\begin{theorem}[Primal-Dual AdaGrad]
\label{thm:pd-adagrad}
Let the iterates be generated by Algorithm~\ref{alg:pd-adagrad} for $T\ge 1$. %
If $g$ is $G$-Lipschitz, then
\begin{equation*}
    g(\bar{\bm{\lambda}}_T)-g_\star=\widetilde{\mathcal{O}}\!\left(\frac{1}{\sqrt{T}}\right),\quad f(\bar{\bm{x}}_T)-f_\star=\widetilde{\mathcal{O}}\!\left(\frac{1}{\sqrt{T}}\right),\quad \dist(\bm{A}\bar{\bm{x}}_T-\bm{b},\mathcal{K})=\widetilde{\mathcal{O}}\!\left(\frac{1}{\sqrt{T}}\right).
\end{equation*}
If instead $g$ is differentiable and $L$-smooth, then
\begin{equation*}
    g(\bar{\bm{\lambda}}_T)-g_\star=\mathcal{O}\!\left(\frac{1}{T}\right),\quad f(\bar{\bm{x}}_T)-f_\star=\mathcal{O}\!\left(\frac{1}{T}\right),\quad \dist(\bm{A}\bar{\bm{x}}_T-\bm{b},\mathcal{K})=\mathcal{O}\!\left(\frac{1}{T}\right).
\end{equation*}
\end{theorem}

\subsection{Primal-Dual AcceleGrad}
\label{subsec:pd-accelegrad}

The accelerated method combines AcceleGrad's weight sequence and adaptive accumulator~\citep{levy2018online} with the dual subgradient and Fenchel-type recovery scheme, as shown in~Algorithm~\ref{alg:pd-accelegrad}. 
For the analysis, we introduce a compact convex dual localization set $\mathcal{D}$ and denote its diameter by $D$. %

\begin{algorithm2e}[t]
\caption{Primal-Dual AcceleGrad}
\label{alg:pd-accelegrad}
\KwIn{Compact convex dual set $\mathcal{D}$, initial point $\bm{\lambda}_0\in\mathcal{D}$, scale parameter $\eta>0$, offset $\delta>0$.}
Initialize $\tilde{\bm{\lambda}}_0=\hat{\bm{\lambda}}_0=\bm{\lambda}_0$ and $\alpha_{-1}=0$\;
\For{$t=0,1,\ldots,T-1$}{
Set $\alpha_t=\big(1+\sqrt{1+4\alpha_{t-1}^2}\big)/2$ and $\tau_t=1/\alpha_t$\;
Set $\bm{\lambda}_t=\tau_t\tilde{\bm{\lambda}}_t+(1-\tau_t)\hat{\bm{\lambda}}_t$\;
Compute $\bm{x}^*_t=\bm{x}^*(\bm{\lambda}_t)$ and $\bm{r}^*_t=\bm{r}^*(\bm{\lambda}_t)$\;
Set $\bm{g}_t=-\bm{A}\bm{x}^*_t+\bm{b}+\bm{r}^*_t$\;
Set $S_t=\delta+\sum_{s=0}^{t}\alpha_s^2\norm{\bm{g}_s}^2$ and $\eta_t=\eta/\sqrt{S_t}$\;
Set $\tilde{\bm{\lambda}}_{t+1}=\proj_{\mathcal{D}}(\tilde{\bm{\lambda}}_t-\alpha_t\eta_t\bm{g}_t)$\;
Set $\hat{\bm{\lambda}}_{t+1}=\bm{\lambda}_t-\eta_t\bm{g}_t$\;
}
\KwOut{$\hat{\bm{\lambda}}_T$ and $\bar{\bm{x}}_T=A_T^{-1}\sum_{t=0}^{T-1}\alpha_t\bm{x}^*_t$, where $A_T=\sum_{t=0}^{T-1}\alpha_t$.}
\end{algorithm2e}

\begin{theorem}[Primal-Dual AcceleGrad]
\label{thm:pd-accelegrad}
Assume $\bm{0}\in\mathcal{D}$, and assume there exists $\rho>0$ such that $\bm{\lambda}_\star+\rho\mathbb{B}\subseteq\mathcal{D}$, where $\mathbb{B}:=\{\bm{u}:\norm{\bm{u}}\leq1\}$.
Let the iterates be generated by Algorithm~\ref{alg:pd-accelegrad} for $T\ge1$. %
If $g$ is $G$-Lipschitz, then
\begin{equation}
    \label{eq:pd-accelegrad-nonsmooth-rate}
    \begin{aligned}
        g(\hat{\bm{\lambda}}_T)-g_\star=\widetilde{\mathcal{O}}\!\left(\frac{1}{\sqrt{T}}\right),\quad\! f(\bar{\bm{x}}_T)-f_\star=\widetilde{\mathcal{O}}\!\left(\frac{1}{\sqrt{T}}\right),\quad\! \dist(\bm{A}\bar{\bm{x}}_T-\bm{b},\mathcal{K})=\widetilde{\mathcal{O}}\!\left(\frac{1}{\sqrt{T}}\right).
    \end{aligned}
\end{equation}
If instead $g$ is differentiable and $L$-smooth, and $\eta>\sqrt{2}D$, then
\begin{equation}
    \label{eq:pd-accelegrad-smooth-rate}
    \begin{aligned}
        g(\hat{\bm{\lambda}}_T)-g_\star=\mathcal{O}\!\left(\frac{1}{T^2}\right),\quad f(\bar{\bm{x}}_T)-f_\star=\mathcal{O}\!\left(\frac{1}{T^2}\right),\quad \dist(\bm{A}\bar{\bm{x}}_T-\bm{b},\mathcal{K})=\mathcal{O}\!\left(\frac{1}{T^2}\right).
    \end{aligned}
\end{equation}
\end{theorem}

The conditions involving $\mathcal{D}$ are inherited from analyses of accelerated adaptive methods in the unconstrained setting~\citep{levy2018online}. 
$\mathcal{D}$ is a theoretical localization device and may, for example, be chosen as a sufficiently large Euclidean ball. 
In practice, we omit the projection onto $\mathcal{D}$ in~Algorithm~\ref{alg:pd-accelegrad}. 

\section{Experiments}
\label{sec:experiments}

We compare our methods with HCGM \citep{yurtsever2018conditional} and CGAL \citep{yurtsever2019conditional}, using its decreasing-step (dcr) and fixed-step (cst) variants. 
These methods similarly rely on inexpensive Fenchel-type oracle evaluations. 
UniPD and AccUniPD \citep{yurtsever2015universal} are omitted because they require a target accuracy $\varepsilon$ in advance. 
For each baseline, we tune its single parameter over powers of $10$. 
Since each method performs one Fenchel-type oracle evaluation per iteration, their per-iteration computational costs are comparable. 
We therefore report performance against the number of iterations.
The experimental code and data are publicly available at \url{https://github.com/KarthikPrakhya/adapd-experiments}.

\begin{figure}[H]
  \centering
  \includegraphics[width=\textwidth]{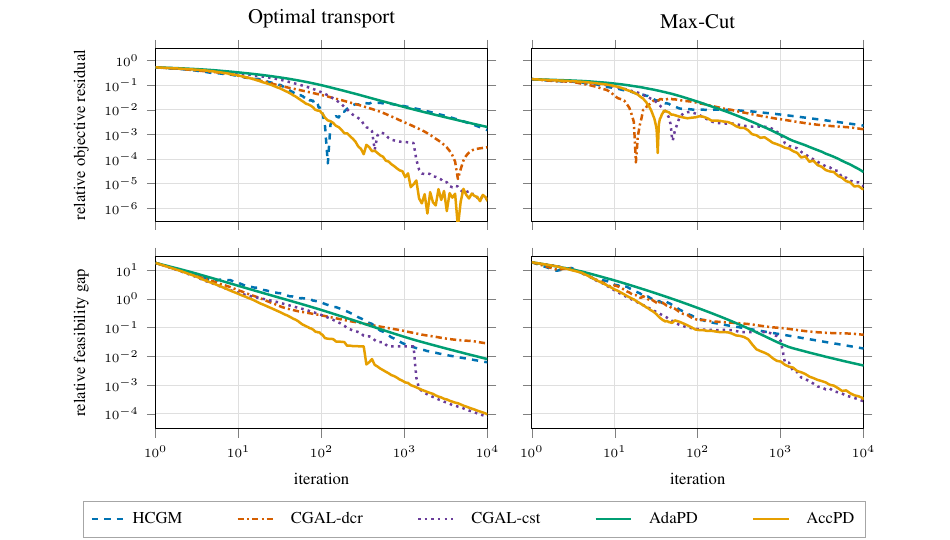}
  \caption{\small Relative primal objective residuals (top) and feasibility gaps (bottom) for OT (left) and G1 Max-Cut (right). Objective residuals are normalized by $\max\{1,|f^\star|\}$ and feasibility gaps by the corresponding problem scale.}
  \label{fig:ot-maxcut-convergence}
\end{figure}

\paragraph{Optimal Transport (OT).}
We consider the discrete Kantorovich optimal transport problem \citep{gabriel2019computational}:
\begin{equation*}
\min_{\bm{X} \in \mathbb{R}^{n \times m}_+} ~  \langle \bm{C}, \bm{X}\rangle 
\quad\mathrm{s.t.}\quad 
\bm{X}\bm{1}_m= \bm{s} ,\quad \bm{X}^\top\bm{1}_n = \bm{d},
\vspace{-0.25em}
\end{equation*}
where $\bm{s}\in\mathbb{R}^n_+$ and $\bm{d}\in\mathbb{R}^m_+$ are the source and target distributions, $\bm{C}\in\mathbb{R}^{n\times m}$ is the transport cost matrix, and $\mathbf{1}_n$ denotes the $n$-dimensional vector of ones.
This is an instance of \eqref{eqn:problem-setup} with 
    $f(\bm{X}) \tighteq \langle\bm{C},\bm{X}\rangle,~    
    \mathcal{A}(\bm{X}) \tighteq \bm{X}\bm{1}_m,~
    \bm{b} \tighteq \bm{s},~
    \mathcal{K} \tighteq \{\bm{0}_n\},~\text{and}~
    \mathcal{X} \tighteq \{\bm{X}\in\mathbb{R}_+^{n\times m}: \bm{X}^\top\bm{1}_n \tighteq \bm{d}\}.$
Here, $\mathcal{X}$ is a product of scaled simplices and the objective is linear, so the Fenchel oracle decomposes columnwise into independent simplex linear minimization oracles that can be evaluated in parallel.

\paragraph{MaxCut SDP.}

We consider the classical SDP relaxation of the weighted Max-Cut problem \citep{goemans1995improved}:
\begin{equation*}
    \max_{\bm{X}\in\mathbb{S}^n_+} ~ \smash{\frac{1}{4}}\langle \bm{L},\bm{X}\rangle    \quad\mathrm{s.t.}\quad
    \mathrm{diag}(\bm{X})=\bm{1}_n,
    \vspace{-0.25em}
\end{equation*}
where $\bm{L}\in\mathbb{S}^n$ is the Laplacian matrix of the weighted graph.
This fits \eqref{eqn:problem-setup} with 
    $f(\bm{X}) \tighteq -\frac{1}{4} \langle\bm{L},\bm{X}\rangle$, 
    $\mathcal{A}(\bm{X}) \tighteq \mathrm{diag}(\bm{X})$,~
    $\bm{b} \tighteq \bm{1}_n$,~
    $\mathcal{K} \tighteq \{\bm{0}_n\},~\text{and}$
    $\mathcal{X} \tighteq \{\bm{X}\in\mathbb{S}^n_+: \mathrm{tr}(\bm{X}) \tighteq n\}.$
The resulting Fenchel oracle is a linear minimization oracle over the scaled spectrahedron, which reduces to an extremal eigenvector computation and can be efficiently approximated using Krylov subspace methods. 

\vspace{0.5em}

For OT, we use the $2000 \!\times\! 2000$ CIFAR-10/MNIST cost matrix from the OTDD pipeline~\cite{alvarez2020geometric}, with median-normalized costs and uniform marginals.
For Max-Cut, we use the $800$-vertex GSet G1 graph~\cite{GSet}. 
Ground truth values are computed with HiGHS for OT and SDPA for Max-Cut~\citep{huangfu2018parallelizing,yamashita2003sdpa,yamashita2010sdpa}. 
The results are shown in Figure~\ref{fig:ot-maxcut-convergence}. 
The observed rates are consistent with our theoretical guarantees. 
Both AdaPD and AccPD improve over the classical nonsmooth rates, while AccPD achieves acceleration. 
Overall, AccPD converges fastest together with the constant-step-size CGAL.

\section{Conclusion}
\label{sec:conclusion}

We introduced adaptive primal-dual subgradient methods that achieve optimal convergence guarantees for both dual and the recovered primal solutions, automatically adapting between nonsmooth and smooth regimes of the dual space. 
Our numerical experiments demonstrate the potential of the proposed methods. 
This is ongoing work, and promising directions include diagonal and matrix-valued adaptive step-sizes, which may better exploit the geometry of ill-conditioned dual problems.

\subsubsection*{Acknowledgments}

This work was supported by the Wallenberg AI, Autonomous Systems and Software Program (WASP) funded by the Knut and Alice Wallenberg Foundation. 
Computational resources were provided by the National Academic Infrastructure for Supercomputing in Sweden (NAISS), funded by the Swedish Research Council. The computational experiments were conducted using resources provided by the High Performance Computing Center North (HPC2N).

\bibliography{./bibliography}

\clearpage

\appendix

\section{Supplementary Analysis}
\label{app:supplementary-analysis}

\subsection{Proofs of Preliminary Results}

\begin{lemma} \label{lem:convex-smooth-ineq-app} Suppose $\phi$ is a differentiable, proper, closed and convex function that is $\gamma$-smooth with respect to a norm $\|\cdot\|$ and that $\bm{\lambda}_\star$ is the minimizer of $\phi$. Then for all $\bm{\lambda} \in \mathbb{R}^n$, the following statement holds:
\begin{equation*}
    \phi(\bm{\lambda}) - \phi(\bm{\lambda}_\star) \geq \frac{1}{2\gamma}\|\nabla \phi(\bm{\lambda})\|^2.
\end{equation*}
\end{lemma}
\begin{proof}
From the smoothness condition, we have
\begin{equation*}
    \phi(\bm{\lambda}_2) \leq \phi(\bm{\lambda}_1) + \langle \nabla \phi(\bm{\lambda}_1), \bm{\lambda}_2 - \bm{\lambda}_1 \rangle + \frac{\gamma}{2}\|\bm{\lambda}_2 - \bm{\lambda}_1\|^2
\end{equation*}
for any $\bm{\lambda}_1, \bm{\lambda}_2 \in \mathbb{R}^n$.
Finding the value of $\bm{\lambda}_2$ that minimizes right-hand side, we get 
\begin{equation*}
    \bm{\lambda}_2 = \bm{\lambda}_1 - \frac{1}{\gamma}\nabla \phi(\bm{\lambda}_1).
\end{equation*}
Substituting the above with $\bm{\lambda}_1 = \bm{\lambda}$ and we get
\begin{equation*}
    \phi(\bm{\lambda}_\star) \leq \phi\left (\bm{\lambda} - \frac{1}{\gamma}\nabla \phi(\bm{\lambda}) \right ) \leq \phi(\bm{\lambda}) - \frac{1}{2\gamma}\|\nabla \phi(\bm{\lambda})\|^2.
\end{equation*}
Rearranging the above proves the desired result.
\end{proof}

\begin{lemma}[AdaGrad sums]
\label{lem:supp-adagrad-sums}
Let $a_t\geq0$, $S_t=\delta+\sum_{\tau=0}^{t}a_\tau$, and $\delta>0$.
Then, for every $T\geq1$,
\begin{equation}
    \label{eq:supp-log-sum}
    \sum_{t=0}^{T-1}\frac{a_t}{S_t}\leq \log\left(\frac{S_{T-1}}{\delta}\right),
\end{equation}
and
\begin{equation}
    \label{eq:supp-sum}
    \sqrt{S_{T-1}}-\sqrt{\delta} 
    \leq \sum_{t=0}^{T-1}\frac{a_t}{\sqrt{S_t}}
    \leq 2\sqrt{S_{T-1}}.
\end{equation}
\end{lemma}

\begin{proof}
Set $S_{-1}:=\delta$.
Since $a_t=S_t-S_{t-1}$ and $S_t\geq S_{t-1}$, we have from the basic inequality $\log x\leq x-1$ with $x:=S_{t-1}/S_t$ that
\begin{equation*}
    \frac{a_t}{S_t}=1-\frac{S_{t-1}}{S_t}\leq \log(S_t)-\log(S_{t-1}).
\end{equation*}
Summing this inequality gives \eqref{eq:supp-log-sum}.
Moreover,
\begin{equation*}
    \frac{a_t}{\sqrt{S_t}}=\frac{S_t-S_{t-1}}{\sqrt{S_t}}
    = \frac{\sqrt{S_t}+\sqrt{S_{t-1}}}{\sqrt{S_t}}(\sqrt{S_t}-\sqrt{S_{t-1}}).
\end{equation*}
Since $1\leq\frac{\sqrt{S_t}+\sqrt{S_{t-1}}}{\sqrt{S_t}}\leq 2$, we have
\begin{equation*}
    \sqrt{S_t}-\sqrt{S_{t-1}}\leq\frac{a_t}{\sqrt{S_t}}\leq2(\sqrt{S_t}-\sqrt{S_{t-1}}).
\end{equation*}
We complete the proof by summing this inequality over $t=0,\dots,T-1$.
\end{proof}

\subsection{Analysis of Primal-Dual AdaGrad}
\label{app:pd-adagrad-analysis}
In this subsection, we provide the analysis for Algorithm~\ref{alg:pd-adagrad} and prove Theorem~\ref{thm:pd-adagrad}. 
To this end, we first introduce the following proposition.
\begin{proposition}[Basic primal-dual recovery]
\label{prop:supp-basic-recovery}
For the iterates generated from Algorithm~\ref{alg:pd-adagrad}, define $A_T=\sum_{t=0}^{T-1}\eta_t$, $\bar{\bm{\lambda}}_T=A_T^{-1}\sum_{t=0}^{T-1}\eta_t\bm{\lambda}_t$, and $\bar{\bm{x}}_T=A_T^{-1}\sum_{t=0}^{T-1}\eta_t\bm{x}^*_t$.
Then
\begin{align}
    g(\bar{\bm{\lambda}}_T)-g_\star &\leq \frac{1}{2A_T}\left(\norm{\bm{\lambda}_0-\bm{\lambda}_\star}^2+\sum_{t=0}^{T-1}\eta_t^2\norm{\bm{g}_t}^2\right), \label{eq:supp-basic-dual} \\
    f(\bar{\bm{x}}_T)-f_\star &\leq \frac{1}{2A_T}\left(\norm{\bm{\lambda}_0}^2+\sum_{t=0}^{T-1}\eta_t^2\norm{\bm{g}_t}^2\right), \label{eq:supp-basic-primal} \\
    \dist(\bm{A}\bar{\bm{x}}_T-\bm{b},\mathcal{K}) &\leq \frac{1}{A_T}\left(2\norm{\bm{\lambda}_0-\bm{\lambda}_\star}+\sqrt{\sum_{t=0}^{T-1}\eta_t^2\norm{\bm{g}_t}^2}\right). \label{eq:supp-basic-feasibility}
\end{align}
\end{proposition}

\begin{proof}
The dual update $\bm{\lambda_{t+1}}=\bm{\lambda}_t-\eta_t\bm{g}_t$ gives, for all $\bm{\lambda}\in\R^n$,
\begin{equation}
    \label{eq:supp-basic-one-step}
    \eta_t\ip{\bm{g}_t}{\bm{\lambda}_t-\bm{\lambda}} = \ip{\bm{\lambda}_t-\bm{\lambda_{t+1}}}{\bm{\lambda}_t-\bm{\lambda}} = \frac{1}{2}\norm{\bm{\lambda}_t-\bm{\lambda}}^2-\frac{1}{2}\norm{\bm{\lambda}_{t+1}-\bm{\lambda}}^2+\frac{\eta_t^2}{2}\norm{\bm{g}_t}^2,
\end{equation}
where we used the identity $2\langle a-b, b-c\rangle = \|a-c\|^2-\|a-b\|^2-\|b-c\|^2$ for the second equality.
Summing \eqref{eq:supp-basic-one-step} over $t=0,\ldots,T-1$ gives
\begin{equation}
    \label{eq:supp-basic-regret}
    \sum_{t=0}^{T-1}\eta_t\ip{\bm{g}_t}{\bm{\lambda}_t-\bm{\lambda}}\leq\frac{1}{2}\norm{\bm{\lambda}_0-\bm{\lambda}}^2+\frac{1}{2}\sum_{t=0}^{T-1}\eta_t^2\norm{\bm{g}_t}^2.
\end{equation}
By convexity and $\bm{g}_t\in\partial g(\bm{\lambda}_t)$,
\begin{equation}\label{eq:g-subgrad}
    g(\bm{\lambda}_t)-g_\star\leq\ip{\bm{g}_t}{\bm{\lambda}_t-\bm{\lambda}_\star}.
\end{equation}
Multiplying \eqref{eq:g-subgrad} by $\eta_t$ and summing over $t=0,\dots,T-1$, then using \eqref{eq:supp-basic-regret} with $\bm{\lambda}=\bm{\lambda}_\star$, and applying Jensen's inequality from convexity of $g$ gives \eqref{eq:supp-basic-dual}.

By the definition of $g$ and $\bm{g}_t$,
\begin{equation}\label{eq:g-f-equality}
    g(\bm{\lambda}_t) 
    = \ip{\bm{\lambda}_t}{-A\bm{x}^*_t}-f(\bm{x}^*_t)+\ip{\bm{\lambda}_t}{b+\bm{r}^*_t}
    = -f(\bm{x}^*_t)+\ip{\bm{g}_t}{\bm{\lambda}_t}.
\end{equation}
Since $g(\bm{\lambda}_t)\geq g_\star=-f_\star$ due to strong duality, we further obtain from \eqref{eq:g-f-equality} that
\begin{equation*}
    f(\bm{x}^*_t)-f_\star\leq\ip{\bm{g}_t}{\bm{\lambda}_t}.
\end{equation*}
Multiplying the above inequality by $\eta_t$ and summing over $t=0,\dots,T-1$, then using \eqref{eq:supp-basic-regret} with $\bm{\lambda}=\bm{0}$, and applying Jensen's inequality from convexity of $f$ gives \eqref{eq:supp-basic-primal}.

For notational convenience, let $\bar{\bm{r}}_T=A_T^{-1}\sum_{t=0}^{T-1}\eta_t\bm{r}^*_t$ and define
\begin{equation} \label{eqn:g_bar_T_expr}
    \bar{\bm{g}}_T:=\frac{1}{A_T}\sum_{t=0}^{T-1}\eta_t\bm{g}_t=-\bm{A}\bar{\bm{x}}_T+\bm{b}+\bar{\bm{r}}_T.
\end{equation}
Since $\bar{\bm{r}}_T\in\mathcal{K}$ from convexity of $\mathcal{K}$, we have $\dist(\bm{A}\bar{\bm{x}}_T-\bm{b},\mathcal{K})\leq\norm{\bar{\bm{g}}_T}$.
The proof of \eqref{eq:supp-basic-primal} before specifying $\bm{\lambda}$ gives, for all $\bm{\lambda}\in\R^n$,
\begin{equation}
    \label{eq:supp-basic-primal-free-lambda}
    f(\bar{\bm{x}}_T)-f_\star\leq\ip{\bar{\bm{g}}_T}{\bm{\lambda}}+\frac{1}{2A_T}\norm{\bm{\lambda}_0-\bm{\lambda}}^2+\frac{1}{2A_T}\sum_{t=0}^{T-1}\eta_t^2\norm{\bm{g}_t}^2.
\end{equation}
From the saddle point inequality, we have
\begin{equation} \label{eqn:saddle-point-inequalty}
    f_\star = \mathcal{L}(\bm{x}_\star, \bm{r}_\star, \bm{\lambda}_\star) \leq \mathcal{L}(\bm{x}, \bm{r}, \bm{\lambda}_\star) = f(\bm{x}) + \langle \bm{\lambda}_\star, \bm{A}\bm{x} - \bm{r} - \bm{b} \rangle,
\end{equation}
for all $\bm{x} \in \mathcal{X}$, $\bm{r} \in \mathcal{K}$. 
Substituting $\bar{\bm{x}}_T$ and $\bar{\bm{r}}_T$ into \eqref{eqn:saddle-point-inequalty} along with \eqref{eqn:g_bar_T_expr} and then rearranging gives 
\begin{equation*}
    \ip{\bar{\bm{g}}_T}{\bm{\lambda}_\star}\leq f(\bar{\bm{x}}_T)-f_\star,
\end{equation*}
which together with \eqref{eq:supp-basic-primal-free-lambda} yields
\begin{equation*}
    \ip{\bar{\bm{g}}_T}{\bm{\lambda}_\star-\bm{\lambda}}-\frac{1}{2A_T}\norm{\bm{\lambda}_0-\bm{\lambda}}^2\leq\frac{1}{2A_T}\sum_{t=0}^{T-1}\eta_t^2\norm{\bm{g}_t}^2.
\end{equation*}
Choosing $\bm{\lambda}=\bm{\lambda}_0-A_T\bar{\bm{g}}_T$ and applying Cauchy-Schwarz inequality gives
\begin{equation*}
    \frac{A_T}{2}\norm{\bar{\bm{g}}_T}^2-\norm{\bar{\bm{g}}_T}\norm{\bm{\lambda}_0-\bm{\lambda}_\star}\leq\frac{1}{2A_T}\sum_{t=0}^{T-1}\eta_t^2\norm{\bm{g}_t}^2.
\end{equation*}
Let $y:=A_T\norm{\bar{\bm{g}}_T}$ and $V_T:=\sum_{t=0}^{T-1}\eta_t^2\norm{\bm{g}_t}^2$.
Then the last inequality is equivalent to
\begin{equation*}
    y^2-2y\norm{\bm{\lambda}_0-\bm{\lambda}_\star}-V_T\leq0,
\end{equation*}
solving which implies $y\leq\norm{\bm{\lambda}_0-\bm{\lambda}_\star}+\sqrt{\norm{\bm{\lambda}_0-\bm{\lambda}_\star}^2+V_T}$.
Since $\sqrt{\norm{\bm{\lambda}_0-\bm{\lambda}_\star}^2+V_T}\leq\norm{\bm{\lambda}_0-\bm{\lambda}_\star}+\sqrt{V_T}$, we have $y\leq2\norm{\bm{\lambda}_0-\bm{\lambda}_\star}+\sqrt{V_T}$.
Since $\dist(\bm{A}\bar{\bm{x}}_T-\bm{b},\mathcal{K})\leq\norm{\bar{\bm{g}}_T}=y/A_T$, this proves \eqref{eq:supp-basic-feasibility}. The proof is now complete.
\end{proof}

\begin{proof}\textbf{of Theorem~\ref{thm:pd-adagrad}.}
Suppose first that $g$ is $G$-Lipschitz. Then $\norm{\bm{g}_t}\leq G$, and hence
\begin{equation*}
    S_t=\delta+\sum_{\tau=0}^{t}\norm{\bm{g}_\tau}^2\leq\delta+G^2T,\qquad t=0,\ldots,T-1.
\end{equation*}
It follows that
\begin{equation*}
    A_T=\sum_{t=0}^{T-1}\frac{\eta}{\sqrt{S_t}}\geq\frac{\eta T}{\sqrt{\delta+G^2T}}.
\end{equation*}
Using Lemma~\ref{lem:supp-adagrad-sums} with $a_t=\norm{\bm{g}_t}^2$ gives
\begin{equation*}
    \sum_{t=0}^{T-1}\eta_t^2\norm{\bm{g}_t}^2
    = \eta^2\sum_{t=0}^{T-1}\frac{\norm{\bm{g}_t}^2}{S_t}
    \leq \eta^2\log\left(\frac{S_{T-1}}{\delta}\right)
    \leq \eta^2\log\left(1+\frac{G^2T}{\delta}\right).
\end{equation*}
Substituting the last two estimates into Proposition~\ref{prop:supp-basic-recovery} yields
\begin{align}
    g(\bar{\bm{\lambda}}_T)-g_\star
    & \leq \frac{\sqrt{\delta+G^2T}}{2\eta T} \left( \norm{\bm{\lambda}_0-\bm{\lambda}_\star}^2 + \eta^2\log\Big(1+\frac{G^2T}{\delta}\Big) \right), \label{eq:supp-pd-adagrad-bd-dual} \\
    f(\bar{\bm{x}}_T)-f_\star
    & \leq \frac{\sqrt{\delta+G^2T}}{2\eta T} \left( \norm{\bm{\lambda}_0}^2 + \eta^2\log\Big(1+\frac{G^2T}{\delta}\Big) \right), \label{eq:supp-pd-adagrad-bd-primal} \\
    \dist(\bm{A}\bar{\bm{x}}_T-\bm{b},\mathcal{K})
    & \leq \frac{\sqrt{\delta+G^2T}}{\eta T} \left( 2\norm{\bm{\lambda}_0-\bm{\lambda}_\star} + \eta \sqrt{\log\Big(1+\frac{G^2T}{\delta}\Big)} \right). \label{eq:supp-pd-adagrad-bd-feasibility}
\end{align}
Omitting constant factors $\eta$, $\delta$, and $G$, $\norm{\bm{\lambda}_0-\bm{\lambda}_\star}$, and $\norm{\bm{\lambda}_0}$, we obtain that the three residuals in \eqref{eq:supp-pd-adagrad-bd-dual}--\eqref{eq:supp-pd-adagrad-bd-feasibility} is of order $\widetilde{\mathcal{O}}(T^{-1/2})$.

Suppose now that $g$ is differentiable and $L$-smooth, and define
\begin{equation*}
    \bar{S} := \max\{\delta,4\eta^2L^2\},
    \qquad
    B := \norm{\bm{\lambda}_0-\bm{\lambda}_\star}^2 + \eta^2\log(\bar{S}/\delta).
\end{equation*}
Since $g$ is also convex, Lemma~\ref{lem:convex-smooth-ineq-app} implies
\begin{equation}\label{eq:adagrad-cvx-1}
    \norm{\bm{g}_t}^2\leq2L(g(\bm{\lambda}_t)-g_\star)\leq2L\ip{\bm{g}_t}{\bm{\lambda}_t-\bm{\lambda}_\star}.
\end{equation}
Set $S_{-1}:=\delta$, and let $t_L$ be the largest index in $\{0,\ldots,T-1\}$ such that $\eta_{t_L}>1/(2L)$, with $t_L:=-1$ if this set is empty.
For $t>t_L$, it follows from \eqref{eq:adagrad-cvx-1} that
\begin{equation*}
    \eta_t^2\norm{\bm{g}_t}^2\leq\frac{\eta_t}{2L}\norm{\bm{g}_t}^2\leq\eta_t\ip{\bm{g}_t}{\bm{\lambda}_t-\bm{\lambda}_\star}.
\end{equation*}
For the prefix $t\leq t_L$, with an empty sum interpreted as zero, Lemma~\ref{lem:supp-adagrad-sums} gives
\begin{equation*}
    \sum_{t=0}^{t_L} \eta_t^2 \norm{\bm{g}_t}^2
    = \eta^2\sum_{t=0}^{t_L} \frac{\norm{\bm{g}_t}^2}{S_t}
    \leq \eta^2 \log \left( \frac{S_{t_L}}{\delta} \right)
    \leq \eta^2 \log(\bar{S}/\delta),
\end{equation*}
where the last inequality holds due to $S_{t_L}=\eta^2/\eta_t^2<4\eta^2L^2\leq \bar{S}$.
Combining the prefix and tail bounds gives
\begin{equation*}
    \sum_{t=0}^{T-1}\eta_t^2\norm{\bm{g}_t}^2
    \leq \eta^2\log(\bar{S}/\delta) + \sum_{t=0}^{T-1}\eta_t\ip{\bm{g}_t}{\bm{\lambda}_t-\bm{\lambda}_\star}.
\end{equation*}
Using \eqref{eq:supp-basic-regret} with $\bm{\lambda}=\bm{\lambda}_\star$ on the last term yields
\begin{equation*}
    \sum_{t=0}^{T-1}\eta_t^2\norm{\bm{g}_t}^2
    \leq \eta^2\log(\bar{S}/\delta) + \frac{1}{2}\norm{\bm{\lambda}_0-\bm{\lambda}_\star}^2 + \frac{1}{2}\sum_{t=0}^{T-1}\eta_t^2\norm{\bm{g}_t}^2,
\end{equation*}
and therefore
\begin{equation}\label{eq:adagrad-cvx-2}
    \sum_{t=0}^{T-1}\eta_t^2\norm{\bm{g}_t}^2 \leq \norm{\bm{\lambda}_0-\bm{\lambda}_\star}^2 + 2\eta^2\log(\bar{S}/\delta).
\end{equation}
Combining the above inequality, \eqref{eq:adagrad-cvx-1}, and \eqref{eq:supp-basic-regret} further give
\begin{equation*}
    \sum_{t=0}^{T-1}\eta_t\norm{\bm{g}_t}^2
    \leq 2L \sum_{t=0}^{T-1}\eta_t\ip{\bm{g}_t}{\bm{\lambda}_t-\bm{\lambda}_\star}
    \leq L\left( \norm{\bm{\lambda}_0-\bm{\lambda}_\star}^2 + \sum_{t=0}^{T-1}\eta_t^2\norm{\bm{g}_t}^2 \right)
    \leq 2LB,
\end{equation*}
which together with $\eta_t=\eta/\sqrt{S_t}$ implies
\begin{equation*}
    \sum_{t=0}^{T-1}\frac{\norm{\bm{g}_t}^2}{\sqrt{S_t}}\leq\frac{2LB}{\eta}.
\end{equation*}
On the other hand, the lower bound in \eqref{eq:supp-sum} gives
\begin{equation*}
    \sqrt{S_t}\leq\sqrt{S_{T-1}}
    \leq \sqrt{\delta}+\sum_{t=0}^{T-1}\frac{\norm{\bm{g}_t}^2}{\sqrt{S_t}}
    \leq \sqrt{\delta}+\frac{2LB}{\eta},~\forall\;0\leq t\leq T-1.
\end{equation*}
Combining this bound and \eqref{eq:adagrad-cvx-2} with Proposition~\ref{prop:supp-basic-recovery} yields
\begin{align}
    g(\bar{\bm{\lambda}}_T)-g_\star
    &\leq \frac{B}{A_T}, \label{eq:supp-pd-adagrad-sm-dual} \\
    f(\bar{\bm{x}}_T)-f_\star
    &\leq \frac{\norm{\bm{\lambda}_0}^2+2B}{2A_T}, \label{eq:supp-pd-adagrad-sm-primal} \\
    \dist(\bm{A}\bar{\bm{x}}_T-\bm{b},\mathcal{K})
    &\leq \frac{1}{A_T}\left( 2\norm{\bm{\lambda}_0-\bm{\lambda}_\star}+\sqrt{2B} \right), \label{eq:supp-pd-adagrad-sm-feasibility}
\end{align}
where
\begin{equation}
    \label{eq:supp-pd-adagrad-sm-AT}
    A_T = \sum_{t = 0}^{T-1} \eta_t = \sum_{t = 0}^{T-1} \frac{\eta}{\sqrt{S_t}} \geq \sum_{t = 0}^{T-1} \frac{\eta}{\sqrt{\delta} + 2LB/\eta} = \frac{\eta T}{\sqrt{\delta} + 2LB/\eta}.
\end{equation}
Since both $\bar{S}$ and $B$ are independent of $T$, \eqref{eq:supp-pd-adagrad-sm-dual}--\eqref{eq:supp-pd-adagrad-sm-AT} show that all three residuals are $\mathcal{O}(T^{-1})$.
\end{proof}

\subsection{Analysis of Primal-Dual AcceleGrad}
\label{app:pd-accelegrad-analysis}

The weights in Algorithm~\ref{alg:pd-accelegrad} are obtained from
\begin{equation}
    \label{eq:supp-alpha-recurrence}
    \alpha_t^2-\alpha_t=\alpha_{t-1}^2,
    \qquad \alpha_{-1}=0,
\end{equation}
which is equivalent to
\begin{equation*}
    \alpha_t = \frac{1+\sqrt{1+4\alpha_{t-1}^2}}{2}, \qquad t \geq 0.
\end{equation*}
In particular, $\alpha_0=1$, and induction gives
\begin{equation}
    \label{eq:supp-alpha-range}
    \frac{t+2}{2} \leq \alpha_t \leq t+1, \qquad t \geq 0.
\end{equation}
Therefore, for every $T\geq1$,
\begin{equation}
    \label{eq:supp-alpha-sums}
    \begin{aligned}
        & A_T := \sum_{t=0}^{T-1} \alpha_t = \sum_{t=0}^{T-1}(\alpha_t^2-\alpha_{t-1}^2) = \alpha_{T-1}^2 \geq \frac{(T+1)^2}{4},\\
        & H_T := \sum_{t=0}^{T-1}\alpha_t^2\leq\sum_{t=1}^{T} t^2=\frac{T(T+1)(2T+1)}{6}.
    \end{aligned}
\end{equation}

\begin{lemma}[Accelerated adaptive regret]
\label{lem:supp-acc-regret}
For the iterates of Algorithm~\ref{alg:pd-accelegrad}, define
\begin{equation*}
    \Delta_t := g(\hat{\bm{\lambda}}_{t+1})-\left( g(\bm{\lambda}_t) + \ip{\bm{g}_t}{\hat{\bm{\lambda}}_{t+1}-\bm{\lambda}_t} + \frac{1}{2\eta_t} \norm{\hat{\bm{\lambda}}_{t+1}-\bm{\lambda}_t}^2 \right).
\end{equation*}
Then, for every $\bm{\lambda}\in\mathcal{D}$,
\begin{equation}
    \label{eq:supp-acc-regret}
    \alpha_{T-1}^2g(\hat{\bm{\lambda}}_T)
    \leq \frac{D^2}{2\eta_{T-1}} + \sum_{t=0}^{T-1}\alpha_t^2\Delta_t + \sum_{t=0}^{T-1}\alpha_t
      \left( g(\bm{\lambda}_t) - \ip{\bm{g}_t}{\bm{\lambda}_t-\bm{\lambda}} \right).
\end{equation}
\end{lemma}

\begin{proof}
The update rules imply
\begin{equation}
    \label{eq:supp-acc-id-1}
    \bm{\lambda}_t-\tilde{\bm{\lambda}}_t = (\alpha_t-1)(\hat{\bm{\lambda}}_t-\bm{\lambda}_t).
\end{equation}
For any $\bm{\lambda}\in\mathcal{D}$, decompose
\begin{equation}
    \label{eq:supp-acc-decomp}
    \alpha_t\ip{\bm{g}_t}{\bm{\lambda}_t-\bm{\lambda}}
    = \alpha_t\ip{\bm{g}_t}{\bm{\lambda}_t-\tilde{\bm{\lambda}}_t} + \alpha_t\ip{\bm{g}_t}{\tilde{\bm{\lambda}}_t-\bm{\lambda}}.
\end{equation}
Using \eqref{eq:supp-acc-id-1} and $\bm{g}_t\in\partial g(\bm{\lambda}_t)$, the first term on the right-hand side of \eqref{eq:supp-acc-decomp} satisfies
\begin{align}
    \label{eq:supp-acc-first-term}
    \alpha_t\ip{\bm{g}_t}{\bm{\lambda}_t-\tilde{\bm{\lambda}}_t}
    & = (\alpha_t^2-\alpha_t) \ip{\bm{g}_t}{\hat{\bm{\lambda}}_t-\bm{\lambda}_t} \notag \\
    & \leq (\alpha_t^2-\alpha_t) \bigl(g(\hat{\bm{\lambda}}_t)-g(\bm{\lambda}_t)\bigr) \notag \\
    & = (\alpha_t^2-\alpha_t) \bigl(g(\hat{\bm{\lambda}}_t)-g(\hat{\bm{\lambda}}_{t+1})\bigr) 
        + (\alpha_t^2-\alpha_t) \bigl(g(\hat{\bm{\lambda}}_{t+1})-g(\bm{\lambda}_t)\bigr).
\end{align}
For the second term, the projection step gives
\begin{equation*}
    \ip{ \tilde{\bm{\lambda}}_t-\alpha_t\eta_t\bm{g}_t - \tilde{\bm{\lambda}}_{t+1}}{\bm{\lambda}-\tilde{\bm{\lambda}}_{t+1}} \leq 0,
    \qquad \forall\,\bm{\lambda}\in\mathcal{D}.
\end{equation*}
Therefore,
\begin{equation}
    \label{eq:supp-proj-var}
    \alpha_t \eta_t \ip{\bm{g}_t}{\tilde{\bm{\lambda}}_{t+1}-\bm{\lambda}}
    \leq \ip{\tilde{\bm{\lambda}}_t-\tilde{\bm{\lambda}}_{t+1}}{\tilde{\bm{\lambda}}_{t+1}-\bm{\lambda}}.
\end{equation}
Using \eqref{eq:supp-proj-var} and the identity $2\ip{a-b}{b-c}=\norm{a-c}^2-\norm{a-b}^2-\norm{b-c}^2$, we obtain
\begin{align}
    \label{eq:supp-acc-second-term}
    \alpha_t\ip{\bm{g}_t}{\tilde{\bm{\lambda}}_t-\bm{\lambda}}
    & \leq \alpha_t \ip{\bm{g}_t}{\tilde{\bm{\lambda}}_t-\tilde{\bm{\lambda}}_{t+1}} + \frac{1}{2\eta_t} \left( \norm{\tilde{\bm{\lambda}}_t-\bm{\lambda}}^2 - \norm{\tilde{\bm{\lambda}}_{t+1}-\bm{\lambda}}^2 - \norm{\tilde{\bm{\lambda}}_t-\tilde{\bm{\lambda}}_{t+1}}^2 \right).
\end{align}
Let $\bm{\mu}_t:=\tau_t\tilde{\bm{\lambda}}_{t+1} + (1-\tau_t)\hat{\bm{\lambda}}_t$.
Then $\bm{\lambda}_t-\bm{\mu}_t = \tau_t(\tilde{\bm{\lambda}}_t-\tilde{\bm{\lambda}}_{t+1})$. 
Since $\tau_t=1/\alpha_t$, the first and last terms in \eqref{eq:supp-acc-second-term} satisfy
\begin{align}
    \label{eq:supp-acc-model-step}
    \alpha_t \ip{\bm{g}_t}{\tilde{\bm{\lambda}}_t-\tilde{\bm{\lambda}}_{t+1}}
      - \frac{1}{2\eta_t} \norm{\tilde{\bm{\lambda}}_t-\tilde{\bm{\lambda}}_{t+1}}^2 
    & = \alpha_t^2\left( \ip{\bm{g}_t}{\bm{\lambda}_t-\bm{\mu}_t} - \frac{1}{2\eta_t}\norm{\bm{\lambda}_t-\bm{\mu}_t}^2 \right)\notag\\
    & \leq \max_{\bm{\mu}\in\mathbb{R}^n} \alpha_t^2\left( \ip{\bm{g}_t}{\bm{\lambda}_t-\bm{\mu}} - \frac{1}{2\eta_t}\norm{\bm{\lambda}_t-\bm{\mu}}^2 \right)\notag\\
    & = \alpha_t^2 \left( \ip{\bm{g}_t}{\bm{\lambda}_t-\hat{\bm{\lambda}}_{t+1}} - \frac{1}{2\eta_t} \norm{\bm{\lambda}_t-\hat{\bm{\lambda}}_{t+1}}^2 \right),
\end{align}
where the last equality uses $\hat{\bm{\lambda}}_{t+1}=\bm{\lambda}_t-\eta_t\bm{g}_t$. 
Combining \eqref{eq:supp-acc-decomp}--\eqref{eq:supp-acc-model-step} with the definition of $\Delta_t$ gives
\begin{align}
    \label{eq:supp-acc-one-step}
    \alpha_t \ip{\bm{g}_t}{\bm{\lambda}_t-\bm{\lambda}}
    & \leq \frac{1}{2\eta_t}\left( \norm{\tilde{\bm{\lambda}}_t-\bm{\lambda}}^2 - \norm{\tilde{\bm{\lambda}}_{t+1}-\bm{\lambda}}^2 \right) \notag \\
        & \quad + (\alpha_t^2-\alpha_t) \bigl(g(\hat{\bm{\lambda}}_t)-g(\hat{\bm{\lambda}}_{t+1})\bigr) + \alpha_t^2\Delta_t - \alpha_t \bigl(g(\hat{\bm{\lambda}}_{t+1})-g(\bm{\lambda}_t)\bigr).
\end{align}
By \eqref{eq:supp-alpha-recurrence}, this simplifies to
\begin{align}
    \label{eq:supp-acc-one-step-simplified}
    \alpha_t\ip{\bm{g}_t}{\bm{\lambda}_t-\bm{\lambda}}
    & \leq \frac{1}{2\eta_t} \left( \norm{\tilde{\bm{\lambda}}_t-\bm{\lambda}}^2 - \norm{\tilde{\bm{\lambda}}_{t+1}-\bm{\lambda}}^2 \right) \notag \\
    & \quad + \alpha_{t-1}^2g(\hat{\bm{\lambda}}_t) - \alpha_t^2 g(\hat{\bm{\lambda}}_{t+1}) + \alpha_t g(\bm{\lambda}_t) + \alpha_t^2 \Delta_t.
\end{align}
Since $S_t$ is nondecreasing, $\eta_t=2D/\sqrt{S_t}$ is nonincreasing. 
Hence, for $\tilde{\bm{\lambda}}_t,\bm{\lambda}\in \mathcal{D}$,
\begin{align}
    & \sum_{t=0}^{T-1}\frac{1}{2\eta_t}\left( \norm{\tilde{\bm{\lambda}}_t-\bm{\lambda}}^2 - \norm{\tilde{\bm{\lambda}}_{t+1}-\bm{\lambda}}^2 \right) \notag \\
    & \quad = \frac{\norm{\tilde{\bm{\lambda}}_0-\bm{\lambda}}^2}{2\eta_0}
      + \sum_{t=1}^{T-1} \left(\frac{1}{2\eta_t}-\frac{1}{2\eta_{t-1}}\right) \norm{\tilde{\bm{\lambda}}_t-\bm{\lambda}}^2 - \frac{\norm{\tilde{\bm{\lambda}}_T-\bm{\lambda}}^2} {2\eta_{T-1}}\leq\frac{D^2}{2\eta_{T-1}}.
    \label{eq:supp-acc-distance-telescope}
\end{align}
Now summing \eqref{eq:supp-acc-one-step-simplified} from $t=0$ to $T-1$ and noting \eqref{eq:supp-acc-distance-telescope} and $\alpha_{-1}=0$, we obtain
\begin{equation*}
    \sum_{t=0}^{T-1}\alpha_t\ip{\bm{g}_t}{\bm{\lambda}_t-\bm{\lambda}}\leq \frac{D^2}{2\eta_{T-1}}-\alpha_{T-1}^2 g(\hat{\bm{\lambda}}_T)+\sum_{t=0}^{T-1}\alpha_t^2\Delta_t+\sum_{t=0}^{T-1}\alpha_t g(\bm{\lambda}_t).
\end{equation*}
Rearranging the above gives the desired result \eqref{eq:supp-acc-regret}.
\end{proof}

\begin{proposition}[Accelerated primal-dual recovery]
\label{prop:supp-acc-recovery}
Suppose $\bm{0}\in\mathcal{D}$, and suppose $\bm{\lambda}_\star+\rho\mathbb{B}\subseteq\mathcal{D}$ for some $\rho>0$.
Then for any $T\ge 1$, we have
\begin{align}
    g(\hat{\bm{\lambda}}_T)-g_\star & \leq\frac{1}{A_T}\left( \frac{D^2}{2\eta_{T-1}} + \sum_{t=0}^{T-1}\alpha_t^2\Delta_t \right), \label{eq:supp-acc-dual} \\
    f(\bar{\bm{x}}_T)-f_\star & \leq \frac{1}{A_T}\left( \frac{D^2}{2\eta_{T-1}} + \sum_{t=0}^{T-1}\alpha_t^2\Delta_t \right), \label{eq:supp-acc-primal} \\
    \dist(\bm{A}\bar{\bm{x}}_T-\bm{b},\mathcal{K}) & \leq \frac{1}{\rho A_T} \left( \frac{D^2}{2\eta_{T-1}} + \sum_{t=0}^{T-1}\alpha_t^2\Delta_t \right). \label{eq:supp-acc-feasibility}
\end{align}
\end{proposition}

\begin{proof}
For any $\bm{\lambda}\in\mathcal{D}$, subtracting $\alpha_{T-1}^2g(\bm{\lambda})$ from both sides of \eqref{eq:supp-acc-regret} and using $A_T=\alpha_{T-1}^2$ gives
\begin{align}
    \alpha_{T-1}^2 \bigl(g(\hat{\bm{\lambda}}_T)-g(\bm{\lambda})\bigr)
    \leq \frac{D^2}{2\eta_{T-1}} + \sum_{t=0}^{T-1} \alpha_t^2\Delta_t+\sum_{t=0}^{T-1}\alpha_t\left( g(\bm{\lambda}_t)-g(\bm{\lambda}) - \ip{\bm{g}_t}{\bm{\lambda}_t-\bm{\lambda}} \right).
    \label{eq:supp-acc-recovery-key}
\end{align}
By convexity and $\bm{g}_t\in\partial g(\bm{\lambda}_t)$, the last sum is nonpositive. 
Hence, taking $\bm{\lambda}=\bm{\lambda}_\star$ in \eqref{eq:supp-acc-recovery-key} proves \eqref{eq:supp-acc-dual}.

By the definitions of $g$ and $\bm{g}_t$,
\begin{equation*}
    g(\bm{\lambda}_t) = -f(\bm{x}_t^*) + \ip{\bm{g}_t}{\bm{\lambda}_t}.
\end{equation*}
Using this identity to rearrange \eqref{eq:supp-acc-regret} gives, for every $\bm{\lambda} \in \mathcal{D}$,
\begin{align}
    \sum_{t=0}^{T-1}\alpha_t f(\bm{x}_t^*)
    & \leq \frac{D^2}{2\eta_{T-1}} + \sum_{t=0}^{T-1}\alpha_t^2\Delta_t - \alpha_{T-1}^2g(\hat{\bm{\lambda}}_T) + \sum_{t=0}^{T-1}\alpha_t\ip{\bm{g}_t}{\bm{\lambda}}\notag\\
    & \leq \frac{D^2}{2\eta_{T-1}} + \sum_{t=0}^{T-1}\alpha_t^2\Delta_t + A_Tf_\star + \sum_{t=0}^{T-1}\alpha_t\ip{\bm{g}_t}{\bm{\lambda}},
    \label{eq:supp-acc-primal-key}
\end{align}
where the last inequality follows from $g(\hat{\bm{\lambda}}_T)\geq g_\star=-f_\star$ and $A_T=\alpha_{T-1}^2$.
Taking $\bm{\lambda}=\bm{0}$ and applying Jensen's inequality proves \eqref{eq:supp-acc-primal}.

For feasibility, let $\bar{\bm{r}}_T=A_T^{-1}\sum_{t=0}^{T-1}\alpha_t\bm{r}^*_t\in\mathcal{K}$ and define
\begin{equation*}
    \bar{\bm{g}}_T := \frac{1}{A_T}\sum_{t=0}^{T-1}\alpha_t\bm{g}_t
    = -\bm{A}\bar{\bm{x}}_T+\bm{b}+\bar{\bm{r}}_T.
\end{equation*}
Applying Jensen's inequality on \eqref{eq:supp-acc-primal-key} gives, for every $\bm{\lambda}\in\mathcal{D}$,
\begin{equation*}
    f(\bar{\bm{x}}_T)-f_\star
    \leq \ip{\bar{\bm{g}}_T}{\bm{\lambda}} + \frac{1}{A_T}\left( \frac{D^2}{2\eta_{T-1}} + \sum_{t=0}^{T-1}\alpha_t^2\Delta_t \right).
\end{equation*}
The saddle-point inequality gives $\ip{\bar{\bm{g}}_T}{\bm{\lambda}_\star} \leq f(\bar{\bm{x}}_T)-f_\star$. 
Thus
\begin{equation*}
    \ip{\bar{\bm{g}}_T}{\bm{\lambda}_\star-\bm{\lambda}}
    \leq \frac{1}{A_T}\left( \frac{D^2}{2\eta_{T-1}} + \sum_{t=0}^{T-1}\alpha_t^2\Delta_t \right),
    \qquad \forall\,\bm{\lambda}\in\mathcal{D}.
\end{equation*}
If $\bar{\bm{g}}_T=\bm{0}$, the feasibility bound is immediate.
Otherwise, choose $\bm{\lambda}=\bm{\lambda}_\star - \rho\bar{\bm{g}}_T/\norm{\bar{\bm{g}}_T}\in\mathcal{D}$.
This gives
\begin{equation*}
    \rho \norm{\bar{\bm{g}}_T}
    \leq \frac{1}{A_T}\left( \frac{D^2}{2\eta_{T-1}} + \sum_{t=0}^{T-1}\alpha_t^2\Delta_t \right).
\end{equation*}
Finally, $\bar{\bm{r}}_T\in\mathcal{K}$ and $\bar{\bm{g}}_T=-\bm{A}\bar{\bm{x}}_T+\bm{b}+\bar{\bm{r}}_T$, so
\begin{equation*}
    \dist(\bm{A}\bar{\bm{x}}_T-\bm{b},\mathcal{K})
    \leq \norm{\bar{\bm{g}}_T} 
    \leq \frac{1}{\rho A_T} \left( \frac{D^2}{2\eta_{T-1}} + \sum_{t=0}^{T-1}\alpha_t^2\Delta_t \right),
\end{equation*}
which proves \eqref{eq:supp-acc-feasibility}. 
The proof is now complete.
\end{proof}

\begin{proof}\textbf{of Theorem~\ref{thm:pd-accelegrad}.}
We first consider the Lipschitz case. 
Since $\hat{\bm{\lambda}}_{t+1}=\bm{\lambda}_t-\eta_t\bm{g}_t$, the model error satisfies
\begin{equation}\label{eq:delta_t-sim}
    \Delta_t=g(\hat{\bm{\lambda}}_{t+1})-g(\bm{\lambda}_t)+\frac{\eta_t}{2}\norm{\bm{g}_t}^2.
\end{equation}
Because $g$ is convex and $G$-Lipschitz,
\begin{equation}
    \label{eq:supp-acc-bd-delta}
    \Delta_t \leq G\norm{\hat{\bm{\lambda}}_{t+1}-\bm{\lambda}_t} +\frac{\eta_t}{2}\norm{\bm{g}_t}^2 = G \eta_t \norm{\bm{g}_t} + \frac{\eta_t}{2}\norm{\bm{g}_t}^2 
    \leq \frac{3}{2} G \eta_t \norm{\bm{g}_t}.
\end{equation}
Let $q_t:=\alpha_t^2\norm{\bm{g}_t}^2$, so that $S_t-S_{t-1}=q_t$ with $S_{-1}:=\delta$.
Using \eqref{eq:supp-acc-bd-delta}, Cauchy--Schwarz, and Lemma~\ref{lem:supp-adagrad-sums} gives (recall definition of $H_T$ in \eqref{eq:supp-alpha-sums})
\begin{align}
    \label{eq:supp-acc-bd-model}
    \sum_{t=0}^{T-1}\alpha_t^2\Delta_t
    & = \frac{3}{2} G \sum_{t=0}^{T-1} \alpha_t^2 \eta_t \norm{\bm{g}_t} 
    = \frac{3G\eta}{2}\sum_{t=0}^{T-1} \alpha_t \frac{\sqrt{q_t}}{\sqrt{S_t}} \notag\\
    & \leq \frac{3G\eta}{2} \sqrt{H_T} \sqrt{\sum_{t=0}^{T-1}\frac{q_t}{S_t}} 
    \leq \frac{3G\eta}{2} \sqrt{H_T} \sqrt{\log\left(\frac{S_{T-1}}{\delta}\right)}.
\end{align}
Moreover, note that
\begin{equation}
    \label{eq:supp-acc-bd-terminal}
    \frac{D^2}{2\eta_{T-1}} = \frac{D^2}{2\eta} \sqrt{S_{T-1}}.
\end{equation}
Since $S_{T-1}\leq\delta+G^2H_T$, combining
\eqref{eq:supp-acc-bd-model} and \eqref{eq:supp-acc-bd-terminal} gives
\begin{equation}
    \label{eq:supp-acc-bd-numerator}
    \frac{D^2}{2\eta_{T-1}} + \sum_{t=0}^{T-1}\alpha_t^2\Delta_t
    \leq \frac{D^2}{2\eta} \sqrt{\delta+G^2H_T} + \frac{3G\eta}{2}\sqrt{H_T} \sqrt{\log\left(1+\frac{G^2H_T}{\delta}\right)}.
\end{equation}
By \eqref{eq:supp-alpha-sums}, $A_T=\Theta(T^2)$ and $H_T=\mathcal{O}(T^3)$. 
Hence $\sqrt{H_T}/A_T=\mathcal{O}(T^{-1/2})$, and Proposition~\ref{prop:supp-acc-recovery} together with \eqref{eq:supp-acc-bd-numerator} proves \eqref{eq:pd-accelegrad-nonsmooth-rate}.

We next consider the smooth case. 
From $L$-smoothness and the definition of $\Delta_t$ in \eqref{eq:delta_t-sim}, we obtain
\begin{align}
    \label{eq:supp-acc-sm-delta}
    \Delta_t & = g(\hat{\bm{\lambda}}_{t+1})-g(\bm{\lambda}_t) - \ip{\bm{g}_t}{\hat{\bm{\lambda}}_{t+1}-\bm{\lambda}_t} - \frac{1}{2\eta_t} \norm{\hat{\bm{\lambda}}_{t+1}-\bm{\lambda}_t}^2 \notag \\
    & \leq \frac{1}{2} \left(L-\frac{1}{\eta_t}\right) \norm{\hat{\bm{\lambda}}_{t+1} - \bm{\lambda}_t}^2
    = \frac{1}{2} (L\eta_t^2-\eta_t) \norm{\bm{g}_t}^2.
\end{align}
Let $t_L:=\max\{0\leq t\leq T-1:\eta_t>1/(2L)\}$ with $t_L:=-1$ if this set is empty and $S_{-1}:=\delta$.
For $t>t_L$, $L\eta_t^2-\eta_t\leq-\eta_t/2$; for $t\leq t_L$, we use $L\eta_t^2-\eta_t\leq L\eta_t^2$.
Therefore,
\begin{align}
    \label{eq:supp-acc-sm-split}
    \frac{D^2}{2\eta_{T-1}} + \sum_{t=0}^{T-1}\alpha_t^2\Delta_t\leq\frac{D^2}{2\eta_{T-1}} + \frac{L}{2}\sum_{t=0}^{t_L}\eta_t^2q_t - \frac{1}{4}\sum_{t=t_L+1}^{T-1}\eta_tq_t.
\end{align}
The first term and negative tail satisfy
\begin{align}
    \label{eq:supp-acc-sm-tail}
    \frac{D^2}{2\eta_{T-1}} - \frac{1}{4} \sum_{t=t_L+1}^{T-1} \eta_t q_t 
    & = \frac{D^2}{2\eta} \sqrt{S_{T-1}} - \frac{\eta}{4} \sum_{t=t_L+1}^{T-1} \frac{q_t}{\sqrt{S_t}} \notag \\
    & \leq \frac{D^2}{2\eta} \sqrt{S_{T-1}} - \frac{\eta}{4} \left(\sqrt{S_{T-1}}-\sqrt{S_{t_L}}\right) 
    \leq \frac{\eta}{4} \sqrt{S_{t_L}},
\end{align}
where the first inequality follows from the lower bound in \eqref{eq:supp-sum}, and the last one holds because $\eta\ge \sqrt{2}D$ implies $\frac{D^2}{2\eta}-\frac{\eta}{4}\leq 0$.
Moreover, note that
\begin{equation}
    \label{eq:supp-acc-sm-prefix}
    \frac{L}{2}\sum_{t=0}^{t_L}\eta_t^2q_t
    = \frac{L\eta^2}{2}\sum_{t=0}^{t_L}\frac{q_t}{S_t}
    \leq \frac{L\eta^2}{2} \log \left(\frac{S_{t_L}}{\delta}\right).
\end{equation}
Applying \eqref{eq:supp-acc-sm-tail} and \eqref{eq:supp-acc-sm-prefix} in \eqref{eq:supp-acc-sm-split} yields
\begin{equation*}
    \frac{D^2}{2\eta_{T-1}} + \sum_{t=0}^{T-1}\alpha_t^2\Delta_t
    \leq \frac{L\eta^2}{2}\log\left(\frac{S_{t_L}}{\delta}\right) + \frac{\eta}{4}\sqrt{S_{t_L}}.
\end{equation*}
If $t_L \geq 0$, then $\eta_{t_L} > 1/(2L)$ implies $S_{t_L} < 4 L^2 \eta^2$; if $t_L = -1$, then $S_{t_L} = \delta$.
Thus, with
\begin{equation*}
    \overline{S}:=\max\{\delta,4L^2\eta^2\},
\end{equation*}
we obtain the finite bound
\begin{equation}
    \label{eq:supp-acc-sm-numerator}
    \frac{D^2}{2\eta_{T-1}} + \sum_{t=0}^{T-1}\alpha_t^2\Delta_t
    \leq \frac{L\eta^2}{2} \log \left(\frac{\overline{S}}{\delta}\right) + \frac{\eta}{4}\sqrt{\overline{S}}.
\end{equation}
Since $A_T\geq(T+1)^2/4$, substituting \eqref{eq:supp-acc-sm-numerator} into Proposition~\ref{prop:supp-acc-recovery} proves \eqref{eq:pd-accelegrad-smooth-rate}.
\end{proof}

\section{Parameter Tuning for Experiments}
\label{sec:parameter_tuning}

For each method and dataset, we vary its scalar parameter over powers of $10$ while keeping the remaining settings fixed. 
The plots below show the relative primal objective residual and relative feasibility gap, normalized as in the main experiments, over the shorter $10^3$-iteration tuning runs.

\newcommand{\includetuningfigure}[3]{%
  \begin{figure}[H]
    \centering
    \includegraphics[width=\textwidth]{#1}
    \caption{#2}
    \label{#3}
  \end{figure}%
}

\includetuningfigure{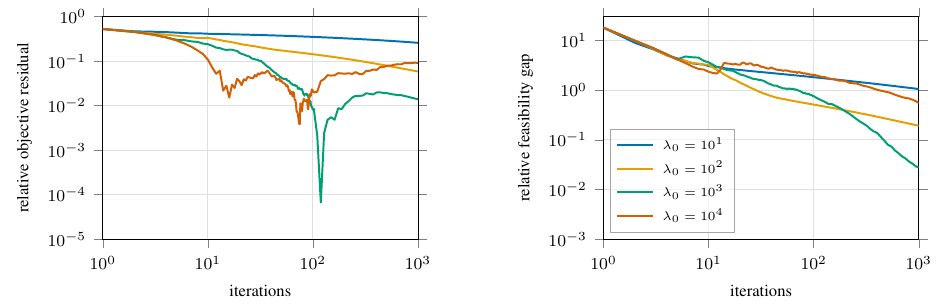}%
  {HCGM tuning on the CIFAR-10/MNIST OT instance.}{fig:ot-HCGM-tuning}
\includetuningfigure{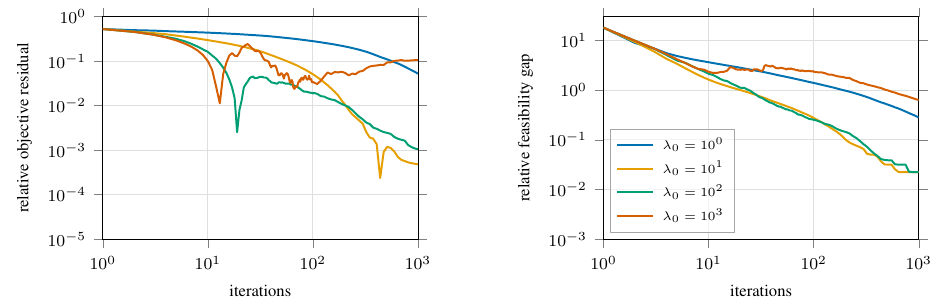}%
  {CGAL-cst tuning on the CIFAR-10/MNIST OT instance.}{fig:ot-CGAL-cst-tuning}
\includetuningfigure{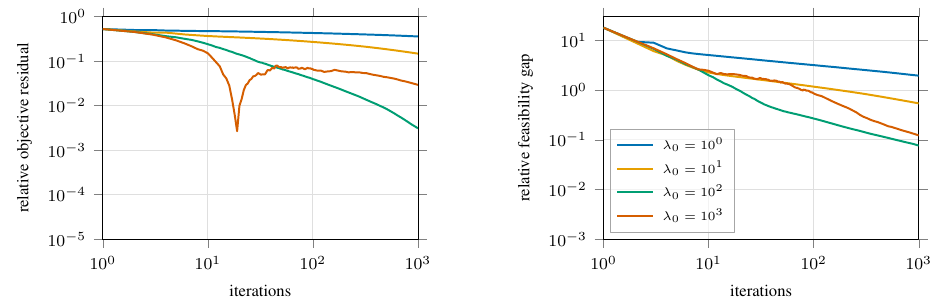}%
  {CGAL-dcr tuning on the CIFAR-10/MNIST OT instance.}{fig:ot-CGAL-dcr-tuning}
\includetuningfigure{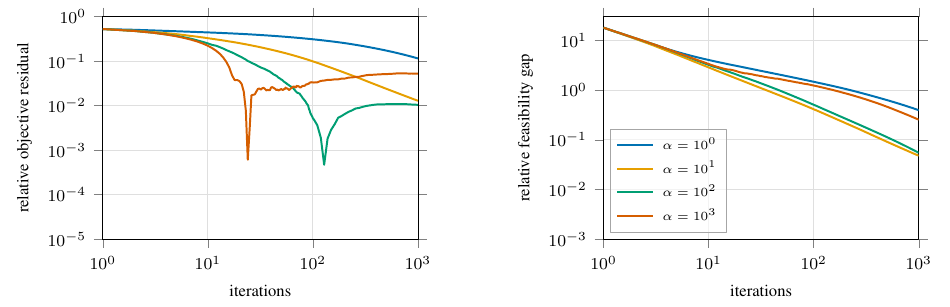}%
  {AdaPD tuning on the CIFAR-10/MNIST OT instance.}{fig:ot-AdaPD-tuning}
\includetuningfigure{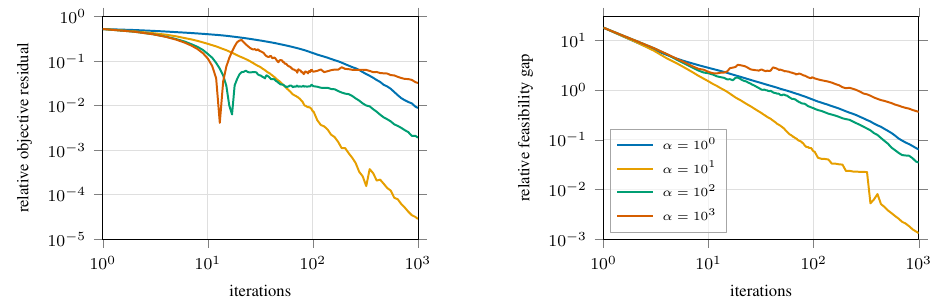}%
  {AccPD tuning on the CIFAR-10/MNIST OT instance.}{fig:ot-AccPD-tuning}

\includetuningfigure{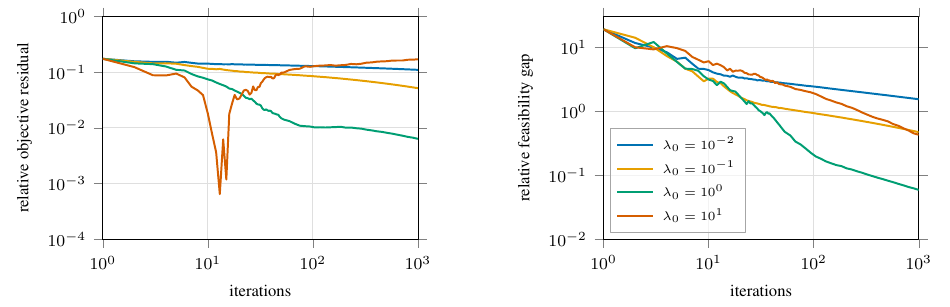}%
  {HCGM tuning on GSet G1.}{fig:HCGM_tuning}
\includetuningfigure{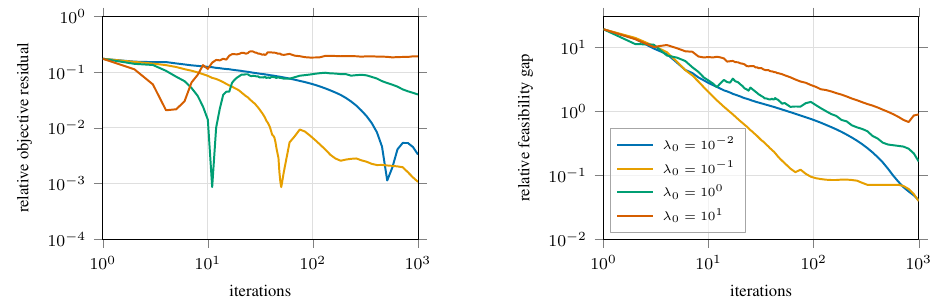}%
  {CGAL-cst tuning on GSet G1.}{fig:CGAL-cst_tuning}
\includetuningfigure{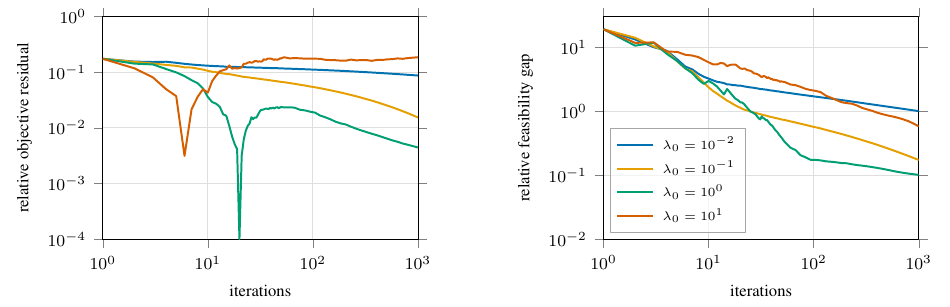}%
  {CGAL-dcr tuning on GSet G1.}{fig:CGAL-dcr_tuning}
\includetuningfigure{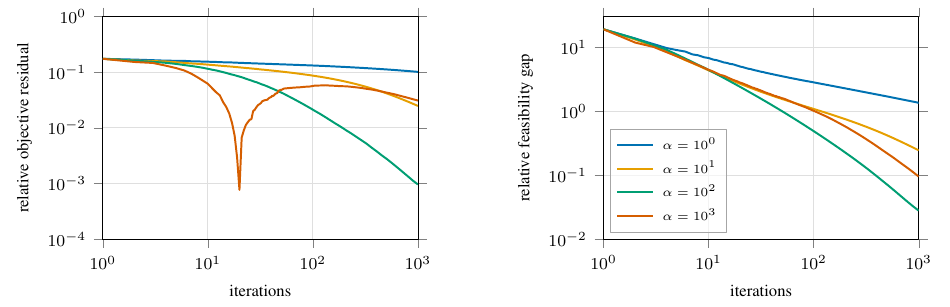}%
  {AdaPD tuning on GSet G1.}{fig:AdaPD_tuning}
\includetuningfigure{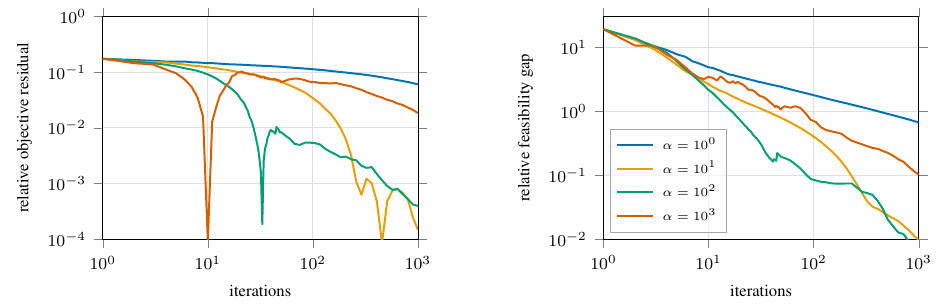}%
  {AccPD tuning on GSet G1.}{fig:AccPD_tuning}
 
\end{document}